\documentclass[11pt,reqno]{amsart}
\usepackage{graphicx}

\usepackage{amscd,amssymb,amsmath,amsthm}
\usepackage{graphicx}
\usepackage{color}
\usepackage{cite}
\newtheorem{thm}{Theorem}
\newtheorem{defn}{Definition}
\newtheorem{lemma}{Lemma}
\newtheorem{pro}{Proposition}
\newtheorem{rk}{Remark}

\numberwithin{equation}{section} 

\def\R{\mathbb{R}}

\begin{document}
\title[Dynamical systems of a gonosomal operator]
{Dynamical systems generated by a gonosomal evolution operator}

\author{U. A. Rozikov, R. Varro}

\address{U.\ A.\ Rozikov\\ Institute of mathematics,
29, Do'rmon Yo'li str., 100125, Tashkent, Uzbekistan.}
\email {rozikovu@yandex.ru}

 \address{R.\ Varro\\Institut de Math\'ematiques et de Mod\'elisation de Montpellier,
Universit\'e de Montpellier, 35095 Montpellier Cedex 5, France.}
\email {richard.varro@univ-montp3.fr}

\begin{abstract}
In this paper we consider discrete-time dynamical systems generated by gonosomal evolution operators of sex linked inheritance.
Mainly we study dynamical systems of a hemophilia, which biologically is a group of hereditary genetic disorders that
impair the body's ability to control blood clotting or coagulation,
which is used to stop bleeding when a blood vessel is broken.
We give an algebraic model of the biological system corresponding
to the hemophilia. The evolution of such system is studied by a nonlinear (quadratic)
gonosomal operator. In a general setting, this operator is considered as a mapping from $\mathbb R^n$,
$n\geq 2$ to itself.  In particular, for a gonosomal operator at $n=4$
we explicitly give all (two) fixed points. Then limit points of the trajectories of the
corresponding dynamical system are studied. Moreover we consider a normalized
version of the gonosomal operator. In the case $n=4$, for the normalized gonosomal operator
we show uniqueness of fixed point and study limit points of the dynamical system.

 \end{abstract}
\maketitle



\section{Introduction.}

In biology it is important a proper understanding of living populations at all levels.
 The relevant mathematics
undoubtedly requires of nonlinear
analysis, in particular a nonlinear dynamical system,
compounded stochastic processes modeling, and the creative
implementation of the computer methodology.

The action of genes is manifested statistically in sufficiently
large communities of matching individuals (belonging to the same
species). These communities are called {\it populations} \cite{L}.
The population exists not only in space but also in time, i.e. it
has its own life cycle. The basis for this phenomenon is
reproduction by mating. Mating in a population can be free or
subject to certain restrictions.

The whole population in space and time comprises discrete
generations $F_0, F_1, \dots$ The generation $F_{n+1}$ is the set of
individuals whose parents belong to the $F_n$ generation. A {\it state} of
a population is a distribution of probabilities of the different
types of organisms in every generation.

A type partition is called
differentiation. The simplest example is sex differentiation. In
bisexual population any kind of differentiation must agree with the
sex differentiation, i.e. all the organisms of one type must belong
to the same sex. Thus, it is possible to speak of male and female
types (see for example \cite{Re}, \cite{LR}, \cite{V} for mathematical models of bisexual population).

In many biological systems, sex is determined genetically: males and females have
different alleles or even different genes that specify their sexual morphology.
In animals, this is often accompanied by {\it chromosomal} differences.
Determination genetically is generally through chromosome combinations
of $XY$ (for example: humans, mammals), $ZW$ (birds), $X0$ (in this variant of
the $XY$ system, females have two copies of the sex chromosome ($XX$) but males have only one ($X0$).
 The $0$ denotes the absence of a second sex chromosome.
There are some sex linked systems which depends on temperature and even some of systems
have sex change phenomenon (see \cite{Ro} for a detailed review.)
 We note that the behavior of sex-linked system can be investigated by studying of nonlinear dynamical systems,
 such systems are not fully understood yet.
A search in MathSciNet gives about 15 mathematical papers which are
related to sex linked models (see for example, \cite{GP}, \cite{K}, \cite{K1}, \cite{R}).
In papers \cite{GR}, \cite{RT} we attempted to introduce
thermodynamic methods in biology. In \cite{V} an algebra associated to a sex change is constructed.

In this paper we consider evolution (dynamical system) of a hemophilia. Recall that
{\it hemophilia} is a group of hereditary genetic disorders that
impair the body's ability to control blood clotting or coagulation,
which is used to stop bleeding when a blood vessel is broken.

In the next section we give a mathematical model of the biological system corresponding
to the hemophilia. The evolution of such system will be given by a nonlinear (quadratic)
evolution operator which is called a gonosomal operator. Thus study of the biological system
is reduced to the study of the nonlinear dynamical system generated by the gonosomal operator.
In a general setting, this operator is considered as a mapping from $\mathbb R^n$,
$n\geq 2$ to itself. In Section 4 we give
some detailed properties of the dynamical system. In particular, for a gonosomal operator at $n=4$
we explicitly give all (two) fixed points. Then limit points of the trajectories of the
corresponding dynamical system are studied.   In the last section we consider the normalized
version of the gonosomal operator. In the case $n=4$, for the normalized gonosomal operator
we show uniqueness of fixed point and study limit points of the dynamical system.

\section{Bisexual population: Gonosomal evolution operator}
Type partition is called
differentiation. The simplest example is sex differentiation. In
bisexual population (BP)  any kind of differentiation must agree with the
sex differentiation, i.e. all the organisms of one type must belong
to the same sex. Thus, it is possible to speak of male and female
types.

In many cases,
the sex determination is genetic, in particular, it is controlled by two chromosomes
called gonosomes. Gonosomal inheritance is a mode of inheritance that is
observed for traits related to a gene encoded on the sex chromosomes.

Let us discuss one example of sex-linked inheritance.
Haemophilia is a lethal recessive $X$-linked disorder:
a female carrying two alleles for hemophilia die.
Therefore if we denote by $X^h$ the gonosome $X$ carrying the
hemophilia, there are only two female genotypes: $XX$ and $XX^h$
($X^hX^h$ is lethal) and two male genotypes: $XY$ and $X^hY$.
We have four types of crosses:
\begin{equation}\label{s4}\begin{array}{llll}
XX \times XY \rightarrowtail {1\over 2}XX, \ \ {1\over 2}XY;\\[2mm]
XX \times X^hY \rightarrowtail {1\over 2}XX^h,\ \ {1\over 2}XY;\\[2mm]
XX^h \times XY \rightarrowtail {1\over 4}XX, {1\over 4}XX^h, {1\over 4}XY, {1\over 4}X^hY;\\[2mm]
XX^h \times X^hY \rightarrowtail {1\over 3}XX^h, {1\over 3}XY, {1\over 3}X^hY.
\end{array}
\end{equation}
Let $F=\{XX, XX^h\}$ and $M=\{XY, X^hY\}$ be sets of genotypes.
Assume state of the set $F$ is given by a real vector $(x,y)$ and state of $M$ by a real vector $(u,v)$.
Then a state of $F\cup M$ is given by the vector $s=(x,y,u,v)\in \mathbb R^4$. If $s'=(x',y',u',v')$ is a state of the system
$F\cup M$ in the next generation then by the rule (\ref{s4}) we get the
evolution operator $W:\mathbb R^4\to\mathbb R^4$ defined by
\begin{equation}\label{W4}
W:\left\{\begin{array}{llll}
x'={1\over 2}xu+{1\over 4}yu\\[2mm]
y'={1\over 2}xv+{1\over 4}yu+{1\over 3}yv\\[2mm]
u'={1\over 2}xu+{1\over 2}xv+{1\over 4}yu+{1\over 3}yv\\[2mm]
v'={1\over 4}yu+{1\over 3}yv.
\end{array}\right.
\end{equation}

This example can be generalized: suppose that the set of
female types is $F=\{1,2,\dots,n\}$ and the set of male types is $M=\{1,2,\dots,\nu \}$.
Let $x=(x_1,\dots,x_n)\in \mathbb R^n$ be a state of $F$ and $y=(y_1,\dots,y_{\nu})\in \mathbb R^{\nu}$ be a state of $M$.

 Consider $\gamma_{ik,j}^{(f)}$ and $\gamma_{ik,l}^{(m)}$ as some inheritance
 real coefficients (not necessary probabilities) with
 \begin{equation}\label{v2}
 \sum_{j=1}^n\gamma_{ik,j}^{(f)}+\sum_{l=1}^\nu \gamma_{ik,l}^{(m)}=1.
\end{equation}

Consider an evolution operator $W:\R^{n+\nu}\to\R^{n+\nu}$ defined as
\begin{equation}\label{v3}
W: \left\{\begin{array}{ll}
x'_j=\sum_{i,k=1}^{n,\nu}\gamma_{ik,j}^{(f)}x_iy_k, \ \ j=1,\dots,n  \\[3mm]
 y'_l=\sum_{i,k=1}^{n,\nu} \gamma_{ik,l}^{(m)}x_iy_k, \ \ l=1,\dots,\nu.
\end{array}\right.
\end{equation}
This operator is called gonosomal evolution operator.
This means that the association $s=(x,y)\in \mathbb R^{n+\nu}\to s'=(x',y')\in \mathbb R^{n+\nu}$
defines a map $W$. The
population evolves by starting from an arbitrary state $s$, then
passing to the state $s'= W(s)$ (in the next 'generation'), then to
the state $s''=W(W(s))$, and so on.
 Thus, states of the population described by the following discrete-time dynamical
 system
 \begin{equation}\label{2}
 s^{(0)},\ \ s^{(1)}= W(s^{(0)}), \ \ s^{(2)}=W^{2}(s^{(0)}),\ \  s^{(3)}= W^{3}(s^{(0)}),\dots
\end{equation}
         where $s^{(0)}\in \mathbb R^{n+\nu}$ is a given initial point and  $W^n(s)=\underbrace{W(W(...W}_n(s))...)$ denotes the $n$ times
         iteration of $W$ to $s$.

  The main problem for a given dynamical system is to
describe the limit points of the trajectory $\{s^{(n)}\}_{n=0}^\infty$ for
arbitrary given $s^{(0)}$.

\section{Dynamical system generated by the operator (\ref{W4})}

Note that operator (\ref{v3}) describes evolution of a hemophilia.
The dynamical systems generated by gonosomal operator (\ref{v3}) is complicated.
In this paper we study the dynamical system generated by gonosomal operator (\ref{W4}), which is a particular case of (\ref{v3}), obtained by
$n=\nu=2$ and the following coefficients:
$$\begin{array}{cccc}
\gamma_{11,1}^{(f)}={1\over 2} &\gamma_{11,2}^{(f)}=0& \gamma_{11,1}^{(m)}={1\over 2}&\gamma_{11,2}^{(m)}=0\\[3mm]
\gamma_{12,1}^{(f)}=0 &\gamma_{12,2}^{(f)}={1\over 2}& \gamma_{12,1}^{(m)}={1\over 2}&\gamma_{12,2}^{(m)}=0\\[3mm]
\gamma_{21,1}^{(f)}={1\over 4} &\gamma_{21,2}^{(f)}={1\over 4}& \gamma_{21,1}^{(m)}={1\over 4}&\gamma_{21,2}^{(m)}={1\over 4}\\[3mm]
\gamma_{22,1}^{(f)}=0 &\gamma_{22,2}^{(f)}={1\over 3}& \gamma_{22,1}^{(m)}={1\over 3}&\gamma_{22,2}^{(m)}={1\over 3}
\end{array}$$

\subsection{Fixed points}

\textcompwordmark{}

A point $s$ is called fixed point if $W(s)=s$.
  Let us find all fixed points of $W$ given by (\ref{W4}), i.e. we solve the following system of equations
\begin{equation}\label{W4f}
W:\left\{\begin{array}{llll}
x={1\over 2}xu+{1\over 4}yu\\[2mm]
y={1\over 2}xv+{1\over 4}yu+{1\over 3}yv\\[2mm]
u={1\over 2}xu+{1\over 2}xv+{1\over 4}yu+{1\over 3}yv\\[2mm]
v={1\over 4}yu+{1\over 3}yv.
\end{array}\right.
\end{equation}
First it is easy to see that $s_0=(0,0,0,0)$ is a solution to system (\ref{W4f}).

To find another solution from the first equation of this system we get $(4-2u)x=yu$. Assuming $u=2$ from this equation we get
$y=0$ then the last equation of (\ref{W4f}) gives $v=0$ consequently from the third equation of the system we get $x=2$.
Thus we obtained the solution $s_2=(2,0,2,0)$.

From the last equation we get $(12-4y)v=3yu$, assume first that $y=3$ then this equation gives $u=0$,
consequently the first equation of the system (\ref{W4f}) gives $x=0$. Then from the second equation we get $v=3$. But these values do not satisfy the third equation of the system. Assume now $u\ne 2$ and $y\ne 3$ then we have
\begin{equation}\label{xv}
x={yu\over 4-2u}, \ \ v={3yu\over 12-4y}.
\end{equation}
Using (\ref{xv}) from the second and third equations of the system (\ref{W4f}) we obtain
$$
\left\{\begin{array}{ll}
16(3-y)(2-u)=3u(8-4u+yu)\\[2mm]
16(3-y)(2-u)=y(24-yu).
\end{array}\right.$$
From the first equation of the last system we find
$$y={12(u^2-6u+8)\over 3u^2-16u+32}.$$
Substituting this to the second equation of the last system we obtain the following
equation
\[
\left(u-2\right)^{2}\left(u-8\right)\left(3u^{2}-14u+24\right)=0.
\]
This equation gives $u=2$, $u=8$  and  $3u^2-14u+24=0$.
For case $u=2$ we have solution $s_2$ mentioned above. The case $u=8$ gives $y=3$ which
does not give solution of the system (\ref{W4f}) as was discussed above.
Thus only remains $3u^2-14u+24=0$ which does not have real solutions.

We proved the following
\begin{pro}\label{pf} The gonosomal operator (\ref{W4}) has exactly two fixed points: $s_0=(0,0,0,0)$ and $s_2=(2,0,2,0)$.
\end{pro}

\subsection{The type of the fixed points}

\textcompwordmark{}

Now we shall examine the type of the fixed points.

\begin{defn} (see \cite{D}). A fixed point $s$ of the operator $W$ is called hyperbolic if
its Jacobian $J$ at $s$ has no eigenvalues on the unit circle.
\end{defn}

\begin{defn} (see \cite{D}). A hyperbolic fixed point $s$ is called:
\begin{itemize}
\item attracting if all the eigenvalues of the Jacobi matrix $J(s)$ are less than 1 in
absolute value;
\item repelling if all the eigenvalues of the Jacobi matrix $J(s)$ are greater than 1 in
absolute value;
\item a saddle otherwise.
\end{itemize}
\end{defn}

To find the type of a fixed point of the operator (\ref{W4}) we write
the Jacobi matrix:
$$J(s)=J_W=\left(\begin{array}{cccc}
{1\over 2}u&{1\over 4}u&{1\over 2}x+{1\over 4}y&0\\[3mm]
{1\over 2}v&{1\over 4}u+{1\over 3}v&{1\over 4}y&{1\over 2}x+{1\over 3}y\\[3mm]
{1\over 2}u+{1\over 2}v&{1\over 4}u+{1\over 3}v&{1\over 2}x+{1\over 4}y&{1\over 2}x+{1\over 3}y\\[3mm]
0&{1\over 4}u+{1\over 3}v&{1\over 4}y&{1\over 3}y
\end{array}\right).$$
It is easy to see that $J(s_0)$ has all eigenvalues equal to 0, therefore
$s_0$ is an attracting point.

The Jacobian $J(s_2)$ has eigenvalues $-{1\over 2}, 0, 1, 2$, therefore the fixed point is not hyperbolic.

\subsection{Dynamics on invariant sets}

\textcompwordmark{}

A set $A$ is called invariant with respect to $W$ if $W(A)\subset A$.

Denote
\begin{eqnarray*}
O & = & \{(0,0,u,v)\in\mathbb{R}^{4}:u,v\in\mathbb{R}\}\cup\{(x,y,0,0)\in\mathbb{R}^{4}:x,y\in\mathbb{R}\},\\
I & = & \{s=(x,y,u,v)\in\mathbb{R}^{4}:y=v=0\},\\
J & = & \{s\in I:x=u\},\\
P & = & \{s=(x,y,u,v)\in\mathbb{R}^{4}:x\geq0,y\geq0,u\geq0,v\geq0\},\\
Q_{a} & = & \{s=(x,y,u,v)\in P:x+y+u+v\leq a\},\qquad a\in[0,4],\\
\mathcal{N} & = & \{s=(x,y,u,v)\in\mathbb{R}^{4}:x\leq0,y\leq0,u\leq0,v\leq0\},\\
\mathcal{N}_{0} & = & \{s=(x,y,u,v)\in\mathbb{R}^{4}:x\leq0,y\leq0,u\geq0,v\geq0\},\\
\mathcal{N}_{1} & = & \{s=(x,y,u,v)\in\mathbb{R}^{4}:x\geq0,y\geq0,u\leq0,v\leq0\}.
\end{eqnarray*}

\begin{lemma}\label{l1}
1) The sets $I$, $J$, $P$ and $Q_a$ ($a\in [0,4]$) are invariant with respect to $W$.

2) $W(O)=\{(0,0,0,0)\}$.

3) $W(Q_a)\subset Q_{a^2/4}$.

4) $W(\mathcal N)\subset P$.

5) $W(\mathcal N_0)\subset \mathcal N$, \ \ $W(\mathcal N_1)\subset \mathcal N$.
\end{lemma}
\begin{proof} We give the proof for $Q_a$, for other sets it simply follows from (\ref{W4}).
Take any $s=(x,y,u,v)\in Q_a$ then we have $0\leq x+y \leq a$ and $0\leq u+v\leq a$.
From (\ref{W4}) we get $x'\geq 0$, $y'\geq 0$, $u'\geq 0$, $v'\geq 0$ and
$$x'+y'+u'+v'=(x+y)(u+v)\leq \left({x+y+u+v\over 2}\right)^2\leq {a^2\over 4}.$$
Thus $s'=(x',y',u',v')\in Q_{a^2\over 4}\subset Q_{a}$.
\end{proof}

Reduce $W$ on $J$ then we get the mapping $x'=f(x)={1\over 2}x^2$. This function has two fixed points $x=0$ and $x=2$. Moreover,
$0$ is attractive ($f'(0)=0<1$) and 2 is repeller ($f'(2)=2>1$). Take an initial point $x_0\in J$ and iterate the function $f$, then we get
$$x_n=f^n(x_0)=2^{-(1+2+2^2...+2^{n-1})}x_0^{2^n}=2^{-2^n+1}x_0^{2^n}=2\left({x_0\over 2}\right)^{2^n}.$$
Hence we have
\[
\lim_{n\to\infty}x_{n}=\left\{ \begin{array}{lll}
0, & \mbox{if} & \left|x_{0}\right|<2\\[2mm]
2, & \mbox{if} & \left|x_{0}\right|=2\\[2mm]
+\infty, & \mbox{if} & \left|x_{0}\right|>2.
\end{array}\right.
\]

Now reduce the operator $W$ on $I$:
\[
V\::\left\{ \begin{array}{lll}
x' & = & \frac{1}{2}xu\\[2mm]
u' & = & \frac{1}{2}xu.\\[2mm]
\end{array}\right.
\]
Thus for any $t_0=(x_0,0,u_0,0)\in I$ we have $W(x_0,0,u_0,0)\in J$.
Consequently, we have full characterization of the dynamical system on the invariant set $J$, i.e., we proved the following
\begin{pro}
For any initial point $t_0=(x_0,0,u_0,0)\in I$ we have
\[
\lim_{n\to\infty}W^{n}(t_{0})=\left\{ \begin{array}{lll}
(0,0,0,0) & \mbox{if} & \left|x_{0}u_{0}\right|<4\\[2mm]
(2,0,2,0) & \mbox{if} & \left|x_{0}u_{0}\right|=4\\[2mm]
+\infty, & \mbox{if} & \left|x_{0}u_{0}\right|>4.
\end{array}\right.
\]
\end{pro}

Let us now consider the dynamical system on the other sets (which may intersect with $J$).

\begin{lemma}\label{l2} Let $a\in [0,4)$. Then for any initial point $s=(x,y,u,v)\in Q_a$ we have
\begin{equation}\label{00}
\lim_{n\to\infty}W^n(s)=(0,0,0,0).
\end{equation}
\end{lemma}
\begin{proof} Let $f(a)=a^2/4$. By Lemma \ref{l1} we have
$$W^n(Q_a)\subset W^{n-1}(Q_{f(a)})\subset W^{n-2}(Q_{f^2(a)})\subset\dots\subset Q_{f^n(a)}.$$
It is easy to see that $f(x)$ has two fixed points $0$ and $4$. Moreover, $0$ is attracting
point and $4$ is repelling point. For any $a\in [0,4)$ we have $\lim_{n\to\infty}f^n(a)=0$.
Consequently,  we get $\lim_{n\to\infty}W^n(Q_a)\subset Q_0=\{(0,0,0,0)\}.$
\end{proof}
\begin{lemma}\label{l3} For an initial point $s=(x,y,u,v)\in Q_4$ the following hold
\begin{itemize}
\item[i.] if there is $k\geq 0$ such that $y^{(k)}v^{(k)}\ne 0$ then (\ref{00}) is satisfied,

\item[ii.] if $y^{(k)}v^{(k)}=0$ for any $k\geq 0$ then
\begin{equation}\label{02}
\lim_{n\to\infty}W^n(s)=(2,0,2,0),
\end{equation}
where $y^{(k)}$ and $v^{(k)}$ are second and fourth coordinates of the vector $W^k(s)$.
\end{itemize}
\end{lemma}
\begin{proof}
 In the case $a=4$ we have $0\leq x+y+u+v\leq 4$. From this inequality it follows that
$$0\leq x+y\leq 2\ \ \mbox{or} \ \  0\leq u+v\leq 2,$$
if both $x+y$ and $u+v$ large than $2$ then their sum is large than 4.
Without loss of generality we assume that $t=x+y\leq 2$ then $u+v\leq 4-t$.
Hence we have
$$ x'+y'+u'+v'=(x+y)(u+v)\leq t(4-t)=\left\{\begin{array}{ll}
4, \ \ \mbox{if} \ \ t=2\\[3mm]
<4, \ \ \mbox{if} \ \ t<2
\end{array}\right.
$$
hence $W(s)\in Q_{t(4-t)}$, i.e., for $t<2$ the case is reduced to the case of $a<4$.
Consider now the case $t=2$. Then if $u+v<2$ we can reduce the case to the case $a<4$.
But if
\begin{equation}\label{x2}
t=x+y=2 \ \ \mbox{and} \ \ u+v=2
\end{equation}
 then from (\ref{W4}) we get
\begin{equation}\label{x4}
x'+y'+u'+v'=4, \ \ x'+y'=2-{yv\over 6} \ \ \mbox{and}   \ \  u'+v'=2+{yv\over 6}.
\end{equation}
Consequently, if $yv\ne 0$ then $W^2(s)\in Q_{4-({yv\over 6})^2}$. By (\ref{x2})
we have $0\leq y\leq 2$ and $0\leq v\leq 2$, hence $0<4-({yv\over 6})^2<4$. Thus condition (\ref{x2}) together
with $yv\ne 0$, by Lemma \ref{l2} gives (\ref{00}).

Let now $yv=0$ then (\ref{x4}) is reduced to the case (\ref{x2}). Repeating above argument we see that if $y'v'\ne 0$ then
$W^3(s)\in Q_{4-({y'v'\over 6})^2}$, otherwise we iterate the argument again. By this way one can show that if
(\ref{x2}) is satisfied and there exists $k\geq 0$ such that $y^{(k)}v^{(k)}\ne 0$ then we have (\ref{00}).

Suppose now (\ref{x2}) is satisfied and
\begin{equation}\label{k}
y^{(k)}v^{(k)}=0 \ \ \mbox{for any} \ \ k\geq 0
\end{equation}
then similarly to (\ref{x4}) we get
 \begin{equation}\label{x5}
x^{(n)}+y^{(n)}+u^{(n)}+v^{(n)}=4, \ \ x^{(n)}+y^{(n)}=2 \ \ \mbox{and}   \ \  u^{(n)}+v^{(n)}=2, \ \ \mbox{for any} \ \ n\geq 0.
\end{equation}
To complete the proof we need to the following
\begin{lemma}\label{l4} If conditions (\ref{x2}) and (\ref{k}) are satisfied then
$$y^{(k)}=v^{(k)}=0 \ \ \mbox{for any} \ \ k\geq 0.$$
\end{lemma}
\begin{proof} From (\ref{W4}) we get
\begin{equation}\label{ky}
\begin{array}{ll}
y^{(k+1)}={1\over 2}x^{(k)}v^{(k)}+v^{(k+1)}\\[3mm]
v^{(k+1)}=\left({1\over 4}u^{(k)}+{1\over 3}v^{(k)}\right)y^{(k)}.
\end{array}
\end{equation}
Now using (\ref{x5}) and (\ref{k}) from (\ref{ky}) we get
\begin{equation}\label{ky1}
\begin{array}{ll}
y^{(k+1)}={1\over 2}(2-y^{(k)})v^{(k)}+v^{(k+1)}=v^{(k)}+v^{(k+1)}\\[3mm]
v^{(k+1)}=\left({1\over 4}(2-v^{(k)})+{1\over 3}v^{(k)}\right)y^{(k)}={1\over 2}y^{(k)}.
\end{array}
\end{equation}
If $v^{(0)}=0$ then from the first equation of (\ref{ky1}) we get $y^{(1)}=v^{(1)}=0$.
Consequently the second equation gives $v^{(2)}=0$. Then using
the first equation we get $y^{(2)}=0$ and so on, we get $y^{(k)}=v^{(k)}=0$ for any $k\geq 0.$

If $y^{(0)}=0$ then the second equation gives $v^{(1)}=0$. Assume $v^{(0)}=v\ne 0$ then from the first equation
we get $y^{(1)}=v+v^{(1)}=v$. Then $v^{(2)}={1\over 2}v$. Consequently, $y^{(2)}={1\over 2}v$. Now condition $y^{(2)}v^{(2)}=0$ gives $v=0$.
This completes the proof.
\end{proof}
Now by Lemma \ref{l4} and property (\ref{x5}) we get
$$x^{(n)}=2 \ \ \mbox{and}   \ \  u^{(n)}=2, \ \ \mbox{for any} \ \ n\geq 0.$$
This completes the proof Lemma \ref{l3}.
 \end{proof}
\begin{lemma}\label{l5} If $s=(x,y,u,v)\in P$ is an initial point with $x+y+u+v>4$, for which
\begin{itemize}
\item[(a)] if there exists $k\geq 0$ such that $(x^{(k)}+y^{(k)})(u^{(k)}+v^{(k)})<4$ then (\ref{00}) is satisfied.
\item[(b)] if $\max\{{xu\over 4}, {yu\over 16}, {yv\over 9}\}>1$ then
$$\lim_{n\to\infty}W^n(s)=\infty,  \ \ \mbox{i.e. at least one coordinate of}\ \ W^n(s) \ \ \mbox{goes to}\ \ \infty.$$
\end{itemize}
\end{lemma}
\begin{proof} (a) This simply follows from the equality
$$x^{(k+1)}+y^{(k+1)}+u^{(k+1)}+v^{(k+1)}=(x^{(k)}+y^{(k)})(u^{(k)}+v^{(k)}).$$
Indeed, from this equality it follows that $W(s^{(k+1)})\in Q_4$. Since $Q_4$ is invariant
the part (a) follows from Lemma \ref{l2}.

(b) Let us prove it for the case $\max\{{xu\over 4}, {yu\over 16}, {yv\over 9}\}={xu\over 4}>1$. For other cases the proof is similar.
  From (\ref{W4}) for any $s=(x,y,u,v)\in P$ we get
\begin{equation}\label{xyv}
x^{(k+1)}\geq {1\over 2}x^{(k)}u^{(k)}, \ \ \ u^{(k+1)}\geq {1\over 2}x^{(k)}u^{(k)}, \ \ k\geq 0.
\end{equation}
Iterating these inequalities we obtain
\begin{equation}\label{xyv1}
x^{(k+1)}\geq\frac{1}{2}x^{(k)}u^{(k)}\geq2^{-(1+2+2^{2}+\dots+2^{k})}(xu)^{2^{k}}=2\left(\frac{xu}{4}\right)^{2^{k}},\quad k\geq0.
\end{equation}
Similarly
\[
u^{(k+1)}\geq2\left(\frac{xu}{4}\right)^{2^{k}},\quad k\geq0.
\]

This completes the proof.
\end{proof}

 Denote
\[
P_{0}=\{s=(x,y,u,v)\in P:\;(x+y)(u+v)<4\},
\]
\[
F=\{s=(x,y,u,v)\in P:\;x+y+u+v>4,\;\max\{\tfrac{xu}{4},\tfrac{yu}{16},\tfrac{yv}{9}\}>1\}.
\]

Summarizing above-mentioned results we get the following

\begin{thm}\label{to1}
If $s=(x,y,u,v)\in \mathbb R^4$ is such that
\begin{itemize}
\item[(i)]
one of the following conditions is satisfied
\begin{itemize}
\item[1)] $s\in P_0$;
\item[2)] $s\in Q_4$ and the condition of part i) of Lemma \ref{l3} is hold;
\item[3)] $s\in \mathcal N$, $W(s)\in P_0$;
\item[4)] $s\in \mathcal N_0$, $W^2(s)\in P_0$;
\item[5)] $s\in \mathcal N_1$, $W^2(s)\in P_0$
\end{itemize}
then
$$\lim_{n\to\infty}W^n(s)=(0,0,0,0).$$
\item[(ii)] one of the following conditions is satisfied
\begin{itemize}
\item[a)] $s\in F$;
\item[b)] $s\in \mathcal N$, $W(s)\in F$;
\item[c)] $s\in \mathcal N_0$, $W^2(s)\in F$;
\item[d)] $s\in \mathcal N_1$, $W^2(s)\in F$
\end{itemize}
then
$$\lim_{n\to\infty}W^n(s)=+\infty.$$

\end{itemize}
\end{thm}
\begin{proof} (i) The case 1) follows from Lemma \ref{l2} and Lemma \ref{l5}. The case 2) is result of Lemma \ref{l3}. By Lemma \ref{l1} we have $W(\mathcal N)\subset P$, $W^2(\mathcal N_0)\subset P$ and $W^2(\mathcal N_1)\subset P$. Consequently, parts 3)-5) follow from 1).

 Part (ii) is a result of Lemma \ref{l5}
and Lemma \ref{l1}.

\end{proof}
\begin{rk} The sum of sets for initial points considered in Theorem \ref{to1} is not equal to $\R^4$.
But for each fixed point
this theorem already gives a large set for the initial point, trajectory of which
converges to the fixed point. Each point $s\in \mathbb R^n$ can be considered as a
state of the system, which is a
generalized measure (or charge) on the set $\{XX,XX^n,XY,X^hY\}$. We considered such measure, because for certain purposes,
it is useful to have a "measure" whose values are not restricted to the non-negative reals or infinity. Moreover, the dynamical
systems considered in this section are interesting because they are higher dimensional and such dynamical
systems are important, but there are relatively few
dynamical phenomena that are currently understood \cite{D}, \cite{GMR}.
In the next section we reduce our operators to the invariant set of vectors
  with non-negative coordinates (the usual measures which take
non-negative values). Then to get a stochastic system of probability measures we use a normalization of the
non-negative measures.
\end{rk}
\section{A normalized gonosomal opertor}
We note that the gonosomal operator (\ref{v3}) does not map the simplex
$$S^{n+\nu-1}=\left\{s=(x_1, \dots, x_n, y_1,\dots,y_\nu)\in \R^{n+\nu}: x_i\geq 0, y_j\geq 0, \sum_{i=1}^nx_i+\sum_{j=1}^\nu y_j=1\right\}$$
to itself, since
\begin{equation}\label{Z}
\sum_{i=1}^nx'_i+\sum_{j=1}^\nu y'_j=\left(\sum_{i=1}^nx_i\right)\left(\sum_{j=1}^\nu y_j\right)
\end{equation}
is not equal to $1$ in general.

We denote
$${\mathcal O}=\left\{s\in S^{n+\nu-1}: (x_1,\dots,x_n)=(0,\dots,0) \, \mbox{or} \, (y_1,\dots,y_\nu)=(0,\dots,0)\right\}.$$
$${\mathcal S}^{n,\nu}=S^{n+\nu-1}\setminus {\mathcal O}.$$
It is easy to see that $W(\mathcal O)=\{(0,\dots,0)\}$. So the points from $\mathcal O$
do not give any contribution to the dynamical system generated by $W$.

Therefore we introduce the normalized gonasomal operator as the following.
Consider the coefficients of the operator (\ref{v3}) with the following properties
\begin{equation}\label{kp}\begin{array}{ll}
\gamma_{ik,j}^{(f)}\geq 0, \ \ \gamma_{ik,l}^{(m)}\geq 0,\\[3mm]
 \sum_{j=1}^n\gamma_{ik,j}^{(f)}+\sum_{l=1}^\nu \gamma_{ik,l}^{(m)}=1, \ \ \mbox{for all} \ \ i,k,j,l.
\end{array}
\end{equation}
An normalized evolution operator $V$, with coefficients (\ref{kp}) is defined as
\begin{equation}\label{v3n}
%\[
V:\left\{ \begin{array}{ll}
x'_{j}\;=\;\dfrac{\sum_{i,k=1}^{n,\nu}\gamma_{ik,j}^{(f)}x_{i}y_{k}}{\Bigl(\sum_{i=1}^{n}x_{i}\Bigr)\left(\sum_{j=1}^{\nu}y_{j}\right)}, & j=1,\dots,n\medskip\\[3mm]
y'_{l}\;=\;\dfrac{\sum_{i,k=1}^{n,\nu}\gamma_{ik,l}^{(m)}x_{i}y_{k}}{\Bigl(\sum_{i=1}^{n}x_{i}\Bigr)\left(\sum_{j=1}^{\nu}y_{j}\right)}, & l=1,\dots,\nu.
\end{array}\right.
%\]
\end{equation}
\begin{pro}\label{pn} The operator $V$ defined by (\ref{v3n}) with coefficients (\ref{kp}) maps ${\mathcal S}^{n,\nu}$ to itself if and only if
the following condition
\begin{equation}\label{ec}
\left(\gamma_{ik,1}^{(f)}, \dots, \gamma_{ik,n}^{(f)}, \gamma_{ik,1}^{(m)}, \dots, \gamma_{ik,\nu}^{(m)}\right)\in {\mathcal S}^{n,\nu}, \ \ \mbox{for all}  \ \ i,k.
\end{equation}
is satisfied.
\end{pro}

\begin{proof} {\sl Necessity.} Suppose for any $s\in {\mathcal S}^{n,\nu}$ we have $s'=V(s)\in {\mathcal S}^{n,\nu}$ then we shall show that (\ref{ec})
is satisfied. Assume that (\ref{ec}) is not true, then there is $i_0\in \{1,\dots,n\}$ and $k_0\in\{1,\dots,\nu\}$ such that
\begin{equation}\label{gx}
\left(\gamma_{i_0k_0,1}^{(f)}, \dots, \gamma_{i_0k_0,n}^{(f)}\right)=(0,\dots,0),
\end{equation}
or
\begin{equation}\label{gy}
\left(\gamma_{i_0k_0,1}^{(m)}, \dots, \gamma_{i_0k_0,\nu}^{(m)}\right)=(0,\dots,0).
\end{equation}
Consider the case (\ref{gx}) (the case (\ref{gy}) is similar). Take now some $s\in {\mathcal S}^{n,\nu}$
such that $x_{i_0}\ne 0$, $x_i=0$ for $i\ne i_0$  and $y_{k_0}\ne 0$, $y_k=0$ for $k\ne k_0$. Then for this $s$ we have
\[
x_{j}'=\frac{\sum_{i,k=1}^{n,\nu}\gamma_{ik,j}^{(f)}x_{i}y_{k}}{\Bigl(\sum_{i=1}^{n}x_{i}\Bigr)\left(\sum_{j=1}^{\nu}y_{j}\right)}=\frac{\gamma_{i_{0}k_{0},j}^{(f)}x_{i_{0}}y_{k_{0}}}{\Bigl(\sum_{i=1}^{n}x_{i}\Bigr)\left(\sum_{j=1}^{\nu}y_{j}\right)}=0,\ \ \mbox{for all}\ \ j=1,\dots,n,
\]
i.e., $s'\in\mathcal O$. This is contradiction to the assumption that
$s'\in \mathcal S^{n,\nu}$.

{\sl Sufficiency.} Assume the conditions (\ref{kp}) and (\ref{ec}) are satisfied, we want to show that if $s\in {\mathcal S}^{n,\nu}$ then
$s'=V(s)\in {\mathcal S}^{n,\nu}$. By the construction of the operator (\ref{v3n}) it is easy to see that $s'\in S^{n+\nu+1}$ so it remains to show
that $s'\notin \mathcal O$. Assume that $s'\in \mathcal O$, i.e, $(x'_1, \dots, x'_n)=(0,\dots,0)$ (the case
 $(y'_1, \dots, y'_\nu)=(0,\dots,0)$ is similar). Then by (\ref{v3n}) we should have
\begin{equation}\label{x0}
\sum_{i,k=1}^{n,\nu}\gamma_{ik,j}^{(f)}x_iy_k=0, \ \ \mbox{for each} \ \  j=1,\dots,n.
\end{equation}
Since $s\in {\mathcal S}^{n,\nu}$ there is $i_0\in \{1,\dots,n\}$  and $k_0\in\{1,\dots,\nu\}$ such that
$x_{i_0}>0$ and $y_{k_0}>0$. From our conditions it follows that $\gamma_{ik,j}^{(f)}x_iy_k\geq 0$,  for all $i,j,k$. Hence from (\ref{x0}) we get
$$\gamma_{i_0k_0,j}^{(f)}x_{i_0}y_{k_0}=0, \ \ \mbox{for each} \ \  j=1,\dots,n,$$
consequently,
$$\gamma_{i_0k_0,j}^{(f)}=0\ \mbox{for each} \ \  j=1,\dots,n.$$
This is contradiction to the condition (\ref{ec}).
\end{proof}

A fixed point $s=(x_1,\dots,x_n,y_1,\dots, y_\nu)$ of the gonosomal operator (\ref{v3}) is called non-negative and normalizeable
if all coordinates of this point are non-negative and $\sum_{i=1}^nx_i+\sum_{k=1}^\nu y_k>0$.

\begin{pro}\label{vr}
There is one-to-one correspondence between non-negative and
normalizeable fixed points of (\ref{v3}) and all fixed points of
(\ref{v3n}).
\end{pro}
\begin{proof}
 Let  $s=(x_1,\dots,x_n,y_1,\dots, y_\nu)$ be a non-negative and normalizeable
fixed point of (\ref{v3}). Denote $Z=\sum_{i=1}^nx_i+\sum_{k=1}^\nu y_k$,
 and consider the point
 $$\tilde{s}=(x_1/Z,\dots,x_n/Z,y_1/Z,\dots, y_\nu/Z).$$
By (\ref{Z}) for the fixed point we have
\begin{equation}\label{Z1}
Z=\left(\sum_{i=1}^nx_i\right)\left(\sum_{k=1}^\nu y_k\right).
\end{equation}
Using formula (\ref{Z1}) one can see that $\tilde{s}$ is a fixed point of (\ref{v3n}).

Now let $\tilde{s}=(\tilde{x}_1,\dots,\tilde{x}_n,\tilde{y}_1,\dots, \tilde{y}_\nu)$ be a
fixed point of (\ref{v3n}), i.e. it satisfies the following system
\begin{equation}\label{v3nf}
\left\{ \begin{array}{ll}
\tilde{x}_{j}\;=\;\dfrac{\sum_{i,k=1}^{n,\nu}\gamma_{ik,j}^{(f)}\tilde{x}_{i}\tilde{y}_{k}}{\Bigl(\sum_{i=1}^{n}\tilde{x}_{i}\Bigr)\left(\sum_{j=1}^{\nu}\tilde{y}_{j}\right)}, & j=1,\dots,n\medskip\\[3mm]
\tilde{y}_{l}\;=\;\dfrac{\sum_{i,k=1}^{n,\nu}\gamma_{ik,l}^{(m)}\tilde{x}_{i}\tilde{y}_{k}}{\Bigl(\sum_{i=1}^{n}\tilde{x}_{i}\Bigr)\left(\sum_{j=1}^{\nu}\tilde{y}_{j}\right)}, & l=1,\dots,\nu.
\end{array}\right.
\end{equation}

Denote
\begin{equation}\label{Z2}
\tilde{Z}=\left(\sum_{i=1}^n\tilde{x}_i\right)\left(\sum_{k=1}^\nu \tilde{y}_k\right).
\end{equation}
Dividing both side of (\ref{v3nf}) to $\tilde{Z}$ it is easy to see that the following point is a fixed point of (\ref{v3}) :
$s=(\tilde{x}_1/\tilde{Z},\dots,\tilde{x}_n/\tilde{Z},\tilde{y}_1/\tilde{Z},\dots, \tilde{y}_\nu/\tilde{Z}).$
\end{proof}

\medskip

In this section we consider the normalized version of the evolution operator (\ref{W4}), i.e.,
\begin{equation}\label{W4n}
%\[
V:\begin{cases}
x'\;= & \dfrac{2xu+yu}{4(x+y)(u+v)}\medskip\\
y'\;= & \dfrac{6xv+3yu+4yv}{12(x+y)(u+v)}\medskip\\
u'\;= & \dfrac{6xu+6xv+3yu+4yv}{12(x+y)(u+v)}\medskip\\
v'\;= & \dfrac{3yu+4yv}{12(x+y)(u+v)}.
\end{cases}
%\]
\end{equation}
It is easy to see that the operator (\ref{W4n}) satisfies the conditions of Proposition \ref{pn}, hence
$V:\mathcal S^{2,2}\to\mathcal S^{2,2}$.

The following lemmas give some useful estimates.

\begin{lemma}\label{lo} Let $s=(x,y,u,v)\in \mathcal S^{2,2}$ and $s^{(1)}=(x',y',u',v')=V(s)$ for the operator (\ref{W4n})
then
\[
\begin{array}{ccc}
\qquad\tfrac{u}{4(u+v)}\;\leq & x' & \leq\;\begin{alignedat}{1}\tfrac{u}{2\left(u+v\right)}\;\leq\;\tfrac{1}{2},\qquad\qquad\end{alignedat}
\medskip\\
\qquad\tfrac{v}{3(u+v)}\;\leq & y' & \leq\;\tfrac{u+2v}{4(u+v)}\;\leq\;\tfrac{1}{2},\qquad\qquad\medskip\\
\tfrac{1}{4}\;\leq\;\tfrac{2x+y}{4(x+y)}\;\leq & u' & \leq\;\tfrac{3x+2y}{6(x+y)}\;\leq\;\tfrac{1}{2},\qquad\qquad\medskip\\
\qquad\tfrac{y}{4(x+y)}\;\leq & v' & \leq\;\tfrac{y}{3(x+y)}\;\leq\;\tfrac{1}{3},\qquad\qquad\medskip\\
\;\tfrac{1}{3}+\tfrac{u}{6(u+v)}\;\leq & x'+y' & \leq\;\tfrac{1}{2},\qquad\qquad\qquad\qquad\medskip\\
\quad\quad\qquad\tfrac{1}{2}\;\leq & u'+v' & \leq\;\tfrac{1}{2}+\tfrac{yv}{6(x+y)(u+v)}\;\leq\;\tfrac{2}{3},\;\;\medskip\\
\quad\quad\qquad v'\;\leq & y' & \leq\;u',\qquad\qquad\qquad\qquad\medskip\\
 & x' & \leq\;u'.\qquad\qquad\qquad\qquad
\end{array}
\]
\end{lemma}
\begin{proof} Straightforward.
\end{proof}

 \begin{lemma}\label{ly} Let $s=(x,y,u,v)\in \mathcal S^{2,2}$ and $s^{(n)}=(x^{(n)},y^{(n)},u^{(n)},v^{(n)})=V^n(s)$ for the operator (\ref{W4n})
then
\begin{itemize}
\item[1.] $${5\over 12}\leq x^{(2)}+y^{(2)}\leq {1\over 2};$$
\item[2.] There exists $\alpha\in (0,1)$ such that
$$v^{(n+1)}\leq \alpha y^{(n)}, \ \ n\geq 2.$$
\end{itemize}
\end{lemma}
\begin{proof} 1. Using Lemma \ref{lo} we have
 $${1\over 3}+{u'\over 6(u'+v')}\leq x^{(2)}+y^{(2)}\leq {1\over 2}.$$
 Consequently, since $u'\geq v'$ we get
 $${1\over 3}+{u'\over 6(u'+v')}\geq {1\over 3}+{1\over 12}={5\over 12}.$$

2. Using $u=1-x-y-v$ we can rewrite the operator (\ref{W4n}) as
\begin{equation}\label{W4s}
V:\left\{\begin{array}{llll}
x'={(2x+y)(1-x-y-v)\over 4(x+y)(1-x-y)}\\[3mm]
y'={3y(1-x-y)+(6x+y)v\over 12(x+y)(1-x-y)}\\[3mm]
v'=y\cdot{3(1-x-y)+v \over 12(x+y)(1-x-y)}.
\end{array}\right.
\end{equation}
For any $n\geq 1$ we have
\begin{equation}\label{Wn}
V^{n+1}\left\{\begin{array}{llll}
x^{(n+1)}={(2x^{(n)}+y^{(n)})(1-x^{(n)}-y^{(n)}-v^{(n)})\over 4(x^{(n)}+y^{(n)})(1-x^{(n)}-y^{(n)})}\\[3mm]
y^{(n+1)}=v^{(n+1)}+\varphi(x^{(n)},y^{n})v^{(n)}\\[3mm]
v^{(n+1)}=\psi(x^{(n)},v^{(n)})y^{(n)},
\end{array}\right.
\end{equation}
where
$$\varphi(x,y)={x\over 2(x+y)(1-x-y)}, \ \ \psi(x,v)={3(1-x-y)+v\over 12(x+y)(1-x-y)}={3u+4v\over 12(u+v)(1-u-v)}.
$$
Using above mentioned inequalities we get
  $$\varphi(x,y)\leq 1, \ \ \psi(x,v)={1\over 3(1-u-v)}-{u\over 12(u+v)(1-u-v)}\leq {2\over 3}-{1\over 8}={13\over 24}.$$

Consequently we get from (\ref{Wn}) the following
  \begin{equation}\label{Wn1}
  \left\{\begin{array}{ll}
y^{(n+1)}\leq v^{(n+1)}+v^{(n)}\\[3mm]
v^{(n+1)}\leq {13\over 24}y^{(n)}.
\end{array}\right.
\end{equation}
This completes the proof.
\end{proof}

\begin{thm}\label{tl} The operator (\ref{W4n}) has a unique fixed point $p=(1/2,0,1/2,0)$ and there is
an open neighborhood $\mathcal{U}(p)\subset \mathcal S^{2,2}$ of $p$ such that for any initial point
$s\in \mathcal U(p)$ we have
$$\lim_{n\to\infty}V^n(s)=p.$$
\end{thm}
\begin{proof} The existence and uniqueness of $p$ follow from Propositions \ref{pf} and  \ref{vr}.
To check the Jacobi matrix
of the operator (\ref{W4n}) at the fixed point $p$, one has to replace $v$ by $v=1-x-y-u$ in the operator
then construct a $3\times 3$ Jacobi matrix. It is easy to see that this matrix at the fixed point has eigenvalues
$-0.5, 0, 1$, i.e the point is attractive, so \cite[Theorem 6.3]{D} completes the proof.
\end{proof}

Using Lemma \ref{lo} and Lemma \ref{ly} one can see that the trajectory of
any initial point after few iterations comes
close to the fixed point $p$. Moreover, using Maple one can see
that the limit point of the trajectory is always $p$.
Thus the following should be true\\

{\bf Conjecture.} For any initial point
$s\in \mathcal S^{2,2}$ we have $\lim_{n\to\infty}V^n(s)=p.$

\begin{rk} The results have the following biological interpretations:
 Let $s=(x,y,u,v)\in \mathcal S^{2,2}$ be an initial state
 (the probability distribution on the set $\{XX, XX^h; XY, X^hY\}$ of genotypes).
 Theorem \ref{tl} says that, as a rule, the population tends to the equilibrium state $p=(1/2,0,1/2,0)$ with the passage
of time, i.e. the future of the population is stable: genotypes $XX$ and $XY$ are survived always, 
but the genotypes $XX^h$ and $X^hY$ (therefore hemophilia) will disappear in the future. It follows that hemophilia is maintained in a population only if it occurs mutations on the genes coding for the coagulation factors.
\end{rk}

\section*{ Acknowledgements}
U.Rozikov thanks Aix-Marseille University Institute for Advanced Study IM\'eRA
(Marseille, France) for support by a residency scheme. His work also partially supported by the Grant No.0251/GF3 of Education and Science Ministry of Republic
of Kazakhstan.

\end{document}